\documentclass{article}
\usepackage[utf8]{inputenc}
\usepackage[english]{babel}
\usepackage{xcolor}
\usepackage{algpseudocode}
\usepackage{algorithm}
\usepackage{array}
\usepackage{multirow}

\usepackage{graphicx}
\usepackage{caption}
\usepackage{subcaption}

\usepackage{amsfonts}
\usepackage{amsthm,amsmath,amssymb}

\usepackage{authblk}

\usepackage{geometry}
\usepackage{amsthm}
\newtheorem{theorem}{Theorem}[section]
\newtheorem{lemma}[theorem]{Lemma}
\newtheorem{corollary}[theorem]{Corollary}

\newtheorem{remark}[theorem]{Remark}

\newtheorem{assumption}{Assumption}[section]

\usepackage{datetime}

\date{\displaydate{date}}

\numberwithin{equation}{section}

\usepackage{dsfont}
\newcommand{\N}{\mathbb{N}}

\let\RE\Re
\let\Re=\undefined
\DeclareMathOperator{\Re}{\RE e}
\let\IM\Im
\let\Im=\undefined
\DeclareMathOperator{\Im}{\IM m}
\newcommand{\norm}[1]{\left\|#1\right\|}
\newcommand{\yd}{y^\delta}

\newcommand{\uk}{u_{k}}
\newcommand{\ukd}{u^{\delta}_{k}}

\newcommand{\ukp}{u_{k+1}}
\newcommand{\ukpd}{u^{\delta}_{k+1}}
\newcommand{\udag}{u^\dagger}

\begin{document}

\title{\LARGE \bf Analysis of Generalized Iteratively Regularized Landweber Iterations driven by data}

\author[1]{Andrea Aspri}
\author[2,3,4]{Otmar Scherzer}
\affil[1]{Department of Mathematics, University of Milan}
\affil[2]{Faculty of Mathematics, University of Vienna}
\affil[3]{Johann Radon Institute for Computational and Applied Mathematics (RICAM)}
\affil[4]{Christian Doppler Laboratory for Mathematical Modeling and Simulation of Next Generations of Ultrasound Devices (MaMSi)}

\vspace{5mm}

\date{}

\maketitle

\thispagestyle{plain}
\pagestyle{plain}

\let\thefootnote\relax\footnotetext{


\thanks{}

}

\begin{abstract}
We investigate generalized versions of the Iteratively Regularized Landweber Method, initially introduced in \cite{Sch98}, to address linear and nonlinear ill-posed problems. Our approach is inspired by the data-driven perspective emphasized in the introduction by Aspri et al. \cite{AspBanOktSch20}. We provide a rigorous analysis establishing convergence and stability results and present numerical outcomes for linear operators, with the Radon transform serving as a prototype.
\end{abstract}
\vskip.15cm

\vskip 10truemm

\section{Introduction}
This article addresses and generalizes an idea mentioned in the introduction of the paper by Aspri et al. \cite{AspBanOktSch20}, for which neither a demonstration nor accompanying numerical results have been provided. The focus is on the convergence of generalized versions of the iteratively regularized Landweber method, initially proposed in \cite{Sch98}, designed for solving linear and nonlinear ill-posed problems represented by the operator equation
\begin{equation} \label{eq:operator}
F(u) = y.
\end{equation}
The analysis centers on operators 
$F:\mathcal{D}(F) \subset X \to Y$, operating between real separable Hilbert spaces $X$ and $Y$, equipped with inner products $\langle\cdot,\cdot\rangle$ 
and norms $\|\cdot\|$, respectively. 
We consider noisy data denoted as $\yd$ with 
\begin{equation*}
\norm{\yd-y}\leq \delta.
\end{equation*} 
Classical research on inverse problems has primarily focused on establishing conditions for the existence and uniqueness of solutions of \eqref{eq:operator}, providing stability estimates and developing stable methods for approximating solutions in the presence of noise, assuming a precise knowledge of the operator $F$. These methods, known as knowledge-driven approaches (see for example \cite{EngHankNeu96,BenBur18,ArrMaaOktSch19, KapSom05}), have been successful but come with limitations. The forward model $F$ is very often an approximate representation of reality, and making it more realistic can be challenging due to a limited understanding of the underlying physical or technical setting. Additionally, computational complexity may restrict the feasibility of accurate analytical models, potentially hindering real-time applications.
In the last years, data-driven approaches, commonly employed in machine learning, provide numerous methods to enhance analytical models, address their inherent limitations and improve some of the existing numerical results, see, for example,  \cite{AdlOek17,AgnColLasMurSanSil20,ArrMaaOktSch19,BenCatRug23,BubBurHelRat23,BubGalLasPraRatSilt21,BubKutLas19,CanSant22,HalNgu23,KalNgu22,Kor22,SchHofNas23} and references therein. 
Current research in the realm of inverse problems is focused on establishing a robust mathematical foundation that integrates data-driven models, especially those rooted in deep learning, with the knowledge-driven ones \cite{ArrMaaOktSch19}, see also, for example, \cite{AlbDeVLasRatSan21,AspBanOktSch20,SchHofNas23}.
In this context, the authors, in \cite{AspBanOktSch20}, propose a data-driven regularization method for solving linear and nonlinear ill-posed problems starting from the iteratively regularized Landweber iteration, firstly studied in \cite{Sch98} (also we refer the reader to \cite{KalNeuSch08}). This method involves iterative updates of the form
\begin{equation}\label{eq:ir_landweber_iter}
\ukp := \uk - F'(\uk)^*\bigl( F(\uk)-\yd \bigr) - \lambda_k(\uk - u^{(0)}),\qquad k\in\mathbb{N},
\end{equation}
where $u^{(0)}$ represents an initial guess that incorporates a priori knowledge about the solution to be recovered. Note that the previous iteration can be equivalentlly rewritten as
\begin{equation}\label{eq:aux0_intr}
\ukp := (1-\lambda_k)\uk - F'(\uk)^*\bigl( F(\uk)-\yd \bigr) + \lambda_k u^{(0)},\qquad k\in\mathbb{N},
\end{equation}
The method operates as a regularization procedure for noisy data when utilizing a posteriori stopping criteria, such as the discrepancy principle, which halts the iteration after the first $k_*=k_*(\delta,\yd)$ steps satisfying:
\begin{equation*}
\norm{F(u_{k_*})-\yd}\leq \tau\delta < \norm{F(\uk)-\yd}, \qquad 0\leq k< k_*, 
\end{equation*}  
for some $\tau>1$. 
In \cite{Sch98}, it has been proven that the modified Landweber iteration converges to a solution which is the closest to $u^{(0)}$ (see also \cite{KalNeuSch08}). 
The damping term $\lambda_k(\uk-u^{(0)})$ was initially incorporated as an extra stabilizing element in Gauß-Newton's method, as detailed in \cite{Bak92}. Subsequently, it was noted that this term exerts a similar stabilizing influence in various iterative regularization methods when employed as an augmenting factor.

Aspri et al. \cite{AspBanOktSch20} explore a variant of the iteratively regularized Landweber method, adopting a data-driven approach. Assuming access to expert data, $(u^{(i)},F(u^{(i)}))_{1 \leq i \leq n}$, they propose a method based on a black-box strategy
\begin{equation}\label{SDschemeNew_II}
\ukp := \uk - F'(\uk)^*\bigl( F(\uk)-\yd \bigr) - {\lambda^{\delta}_k} \hat{A}'(\uk) ^*\bigl(\hat{A}(\uk) -\yd \bigr), \qquad k\in\N.
\end{equation}
Here, the operator $\hat{A}$ is handcrafted, mapping each $u^{(i)}$ to $F(u^{(i)})$, $i=1,\ldots,n$, and vice versa. The purpose of the second term in \eqref{SDschemeNew_II} is to introduce a bias for the expert data. It is noteworthy that \emph{system identification} and \emph{black-box} theory received extensive consideration in the sixties of the last century, as documented, for instance, in \cite{Pap62}.

This paper explores a generalized version of \eqref{eq:ir_landweber_iter} that we call Generalized Iteratively Regularized Landweber Iteration (GIRLI), which disregards image data $(F(u^{(i)})_{1 \leq i \leq n})$, and takes into account only the data $u^{(i)}$, for $i=1,\cdots,n$. To be more precise, taking inspiration from \eqref{eq:aux0_intr}, we explore two distinct approaches that consider either the average or the geometric mean of the data $u^{(i)}$, that is
\begin{equation*}
\ukp := (1-\lambda_k)\uk - F'(\uk)^*\bigl( F(\uk)-\yd \bigr) + \frac{\lambda_k}{n}\sum_{i=1}^{n}u^{(i)},\qquad k\in\mathbb{N},
\end{equation*}
and
\begin{equation*}
\ukp := (1-\lambda_k)\uk - F'(\uk)^*\bigl( F(\uk)-\yd \bigr) + \lambda_k\sqrt[n]{\prod_{i=1}^{n} u^{(i)}},\qquad k\in\mathbb{N},
\end{equation*}
both complemented by an additional initial guess denoted by $u_0$.
We refer to the second method as Generalized Iteratively Regularized Landweber Iteration with Geometric Mean (GIRLI-GM).
\\
Taking inspiration from \cite{Sch98,KalNeuSch08}, the main objective of the paper is to prove the convergence of both the above methods to the solution of the operator equation \eqref{eq:operator} closest to the average
$\frac{1}{n}\sum_{i=1}^{n}u^{(i)}$, as stated in \cite{AspBanOktSch20}, and to $\sqrt[n]{\prod_{i=1}^{n} u^{(i)}}$, respectively. The study establishes strong convergence and stability results under assumptions of Fréchet differentiability and the usual tangential cone condition for $F$. Numerical simulations, focusing on the Radon transform, complement the theoretical analysis, comparing various regularization methods, including GIRLI, the data-driven iteratively regularized Landweber scheme \eqref{SDschemeNew_II}, the classic iteratively regularized Landweber scheme \eqref{eq:ir_landweber_iter}, and the Landweber iteration.

The outline of the paper is as follows: In Section \ref{sec:GLI} we analyze GIRLI iterates in an infinite dimensional Hilbert 
space setting and prove strong convergence and stability. In Section \ref{sec:NM} we discuss the numerical implementation of GIRLI and 
study applications to the Radon transform (for the linear case) with full and limited data. Moreover, we compare the results with other iterative methods. In Section \ref{sec:conclusion}, we present our conclusions.

\section{Generalized Iteratively Regularized Landweber Iteration}\label{sec:GLI}
In this section, we analyse the convergence behavior of the generalized iteratively regularized Landweber iterations, as introduced in the previous section. We assume to have a set of data $u^{(i)}$, for $i=1,\cdots,n$ similar to what is intended to be reconstructed, and therefore, can be utilized as prior information in the reconstruction procedure. 
For clarity, we restate the GIRLI iterations below
\begin{equation}\label{eq:GIRLI}
\begin{aligned}
&\ukpd := (1-\lambda_k)\ukd - F'(\ukd)^*\bigl( F(\ukd)-\yd \bigr) + \frac{\lambda_k}{n}\sum_{i=1}^{n}u^{(i)},\qquad k\in\mathbb{N},\\
&u^{\delta}_0=u_0 - \textrm{initial guess},
\end{aligned}
\end{equation}
and
\begin{equation}\label{eq:GIRLI-GM}
\begin{aligned}
&\ukpd := (1-\lambda_k)\ukd - F'(\ukd)^*\bigl( F(\ukd)-\yd \bigr) + \lambda_k \sqrt[n]{\prod_{i=1}^{n}u^{(i)}},\qquad k\in\mathbb{N},\\
&u^{\delta}_0=u_0 - \textrm{initial guess},
\end{aligned}
\end{equation}
where $\yd$ represents noisy data, satisfying 
\begin{equation*}
\norm{\yd-y}\leq \delta,
\end{equation*}
and ${\lambda_k}$ is a suitably chosen parameter.\\	
In the sequel, we assume that \eqref{eq:operator} admits a solution $\udag$, which may not be not unique.

For our analysis, we follow the methodology outlined in \cite[Chapters 3, Section 2]{KalNeuSch08} and \cite{Sch98}, given the modified nature of the Landweber scheme in our iterates.\\
We use the notation $u_k=\ukd$, when $\delta=0$.

In the following, we establish some assumptions regarding the distance between a solution of \eqref{eq:operator} and the average and mean geometric terms, the operator $F$ and the damping factor $\lambda_k$ in order to ensure the local convergence of the iterative scheme \eqref{eq:GIRLI}. Additionally, we define some specific constants pivotal to the proof of the convergence theorem. 

Let us denote by $B_{r}(u)$ a ball of center $u$ and radius $r$.
\begin{assumption}\label{as:average_mean_geometric}
Let us assume that there exists $\udag\in B_{\rho}(u_0)$, with $\rho>0$ and $\udag$ solution to the operator equation \eqref{eq:operator}, such that 
\begin{enumerate}
 \item[(a)] $\norm{\udag-\frac{1}{n}\sum_{i=1}^{n}u^{(i)}}< \rho$;
 \item[(b)] $\norm{\udag-\sqrt[n]{\prod_{i=1}^{n}u^{(i)}}}< \rho$.
\end{enumerate}
\end{assumption}
Moreover, we also assume the following hypotheses. 
\begin{assumption}\label{as:fa}
Given $r>\rho$ (which will be specified later), we require that
	\begin{enumerate}
		\item $F$ satisfies the \emph{tangential cone condition}, that is 
		\begin{equation}\label{eq:tcc}
		\norm{F(u) - F(v) - F'(u)(u-v)} \leq \eta \norm{F(u) - F(v)}, \qquad         \forall u,v \in \mathcal{B}_{r}(u_0),
		\end{equation}
		where $0<\eta<\frac{1}{2}$.		 	
		\item $F$ has a continuous Fr\'echet derivative $F'$ with Lipschitz constants $L$, i.e., 
			\begin{equation}\label{eq:lcF}
			\norm{F'(u)} \leq L,\qquad \forall u \in \mathcal{B}_{r}(u_0).
			\end{equation} 
            \item The damping factor $\lambda_k$ is such that 
            \begin{equation}\label{eq:lambdak}
            0<\lambda_k\leq \lambda_{max}< 1.  
            \end{equation}
	\end{enumerate}
\end{assumption}
We commence by presenting the convergence analysis for the iteration defined in \eqref{eq:GIRLI}. Subsequently, the identical theoretical framework is extended to the case of \eqref{eq:GIRLI-GM}.

Here, we show that, under certain conditions, the iterations of \eqref{eq:GIRLI} remain in a ball around $u_0$. 
\begin{lemma}\label{prop:suff_cond}
    Under the Assumptions \ref{as:average_mean_geometric} (a) - \ref{as:fa}, let $\kappa$ be a positive constant such that $\kappa\in(0,1)$ and
    \begin{equation}\label{eq:crho}
        c(\rho):=\rho\frac{1-\lambda_{max}+\sqrt{1+\lambda_{max}(2-\lambda_{max})\frac{L^2}{\kappa^{2}}}}{2-\lambda_{max}}.
    \end{equation}
Let $r=\rho+c(\rho)$ in \eqref{eq:tcc}-\eqref{eq:lcF}, and $\eta$ satisfy the following inequality
\begin{equation}\label{eq:E}
    E:=2-L^2-2\eta-2\lambda_k(1-\eta)-\kappa^2>0.
\end{equation}
If $\ukd\in B_{c(\rho)}(\udag)$, then a sufficient condition for $\ukpd\in B_{c(\rho)}(\udag)\subset B_{\rho+c(\rho)}(u_0)$ is that
\begin{equation}\label{eq:suff_cond}
    \norm{F(\ukd)-\yd}\geq 2(1-\lambda_{max})\frac{1+\eta}{E}\delta.
\end{equation}
\end{lemma}
\begin{proof}
    Let us consider $\norm{\ukpd-\udag}^2$. Using \eqref{eq:GIRLI} and adding and subtracting properly $\lambda_k \udag$, we find that 
    \begin{equation}\label{eq:norm_ukp1muk}
    \begin{aligned}
        \norm{\ukpd-\udag}^2&=(1-\lambda_k)^2\norm{\ukd-\udag}^2+\frac{\lambda^2_k}{n^2}\norm{\sum_{i=1}^{n}(\udag-u^{(i)})}^2+\norm{F'(\ukd)^*(F(\ukd)-\yd)}^2\\
        &+\frac{2\lambda_k}{n}(1-\lambda_k)\left\langle\ukd-\udag,\sum_{i=1}^{n}(u^{(i)}-\udag)\right\rangle\\
        &+\frac{2\lambda_k}{n}\left\langle \sum_{i=1}^{n}(\udag-u^{(i)}),F'(\ukd)^*(F(\ukd)-\yd)\right\rangle\\
        &+2(1-\lambda_k)\left\langle F'(\ukd)(\ukd-\udag),\yd-F(\ukd)\right\rangle.
    \end{aligned}
    \end{equation}
Note that 
\begin{equation*}
\begin{aligned}
    \langle\yd-F(\ukd),F'(\ukd)(\ukd-\udag)\rangle&=\langle \yd-F(\ukd), \yd-y\rangle-\norm{\yd-F(\ukd)}^2\\
    &-\left\langle \yd-F(\ukd),F(\ukd)-F(\udag)-F'(\ukd)(\ukd-\udag)\right\rangle,
\end{aligned}
\end{equation*}
hence, by the tangential cone condition \eqref{eq:tcc} (note that $\ukd,\ \udag  \in B_{\rho+c(\rho)}(u_0)$), and the fact that $\norm{F(\ukd)-F(\udag)}\leq \norm{F(\ukd)-\yd}+\norm{\yd-y}$, we get
\begin{equation}\label{eq:est_mixed_term1}
\begin{aligned}
    2(1-\lambda_k)\Big|\langle\yd-F(\ukd)&,F'(\ukd)(\ukd-\udag)\rangle\Big|\\
    &\leq 2(1-\lambda_k)\norm{\yd-F(\ukd)}\left((1+\eta)\delta-(1-\eta)\norm{\yd-F(\ukd)}\right).
\end{aligned}
\end{equation}
Additionaly, since $\ukd\in B_{\rho+c(\rho)}(u_0)$, using the assumption that $\norm{F'(u)}\leq L$, for all $u\in B_{\rho+c(\rho)}(u_0)$, we find
\begin{equation}\label{eq:est_mixed_term2}
\begin{aligned}
    \Bigg|\frac{2\lambda_k}{n}&\left\langle \sum_{i=1}^{n}(\udag-u^{(i)}),F'(\ukd)^*(F(\ukd)-\yd)\right\rangle\Bigg|\\
    &= \frac{2\lambda_k}{n}\Bigg|\left\langle F'(\ukd)\left(\sum_{i=1}^{n}(\udag-u^{(i)})\right),F(\ukd)-\yd\right\rangle\Bigg|\\
    &\leq \frac{2\lambda_k}{n} \norm{F'(\ukd)}\norm{\sum_{i=1}^{n}(\udag-u^{(i)})} \norm{F(\ukd)-\yd}\\
    &\leq \frac{2\lambda_k L}{n} \norm{\sum_{i=1}^{n}(\udag-u^{(i)})} \norm{F(\ukd)-\yd}\\
    &\leq \frac{\lambda^2_kL^2}{n^2\kappa^2} \norm{\sum_{i=1}^{n}(\udag-u^{(i)})}^2+\kappa^2 \norm{F(\ukd)-\yd}^2.
\end{aligned}
\end{equation}
Using \eqref{eq:est_mixed_term1} and \eqref{eq:est_mixed_term2} into \eqref{eq:norm_ukp1muk}, we get
\begin{equation}\label{eq:aux1}
\begin{aligned}
    \norm{\ukpd-\udag}^2&\leq(1-\lambda_k)^2\norm{\ukd-\udag}^2+\frac{\lambda^2_k}{n^2}\norm{\sum_{i=1}^{n}(\udag-u^{(i)})}^2+L^2\norm{(F(\ukd)-\yd)}^2\\
        &+\frac{2\lambda_k}{n}(1-\lambda_k)\norm{\ukd-\udag}\norm{\sum_{i=1}^{n}(u^{(i)}-\udag)}\\
        &+2(1-\lambda_k)\norm{\yd-F(\ukd)}\left((1+\eta)\delta-(1-\eta)\norm{\yd-F(\ukd)}\right)\\
        &+\frac{\lambda^2_kL^2}{n^2\kappa^2} \norm{\sum_{i=1}^{n}(\udag-u^{(i)})}^2+\kappa^2 \norm{F(\ukd)-\yd}^2\\
        &=(1-\lambda_k)^2\norm{\ukd-\udag}^2 + \frac{\lambda^2_k}{n^2}\left(1+\frac{L^2}{\kappa^2}\right)\norm{\sum_{i=1}^{n}(\udag-u^{(i)})}^2\\
        &+\frac{2\lambda_k}{n}(1-\lambda_k)\norm{\ukd-\udag}\norm{\sum_{i=1}^{n}(u^{(i)}-\udag)}\\
        &-\norm{F(\ukd)-\yd}\left(2(1-\lambda_k)(1+\eta)\delta-E\norm{F(\ukd)-\yd}\right)
     \end{aligned}
     \end{equation}
     that is 
\begin{equation*}
\begin{aligned}
        \norm{\ukpd-\udag}^2&\leq (1-\lambda_k)^2\norm{\ukd-\udag}^2 + \frac{\lambda^2_k}{n^2}\left(1+\frac{L^2}{\kappa^2}\right)\norm{\sum_{i=1}^{n}(\udag-u^{(i)})}^2\\
        &+\frac{2\lambda_k}{n}(1-\lambda_k)\norm{\ukd-\udag}\norm{\sum_{i=1}^{n}(u^{(i)}-\udag)},
\end{aligned}
\end{equation*}
where the last inequality comes from \eqref{eq:E} and \eqref{eq:suff_cond}. Then, thanks to Assumption \eqref{as:average_mean_geometric} (a) and the fact that $\norm{\ukd-\udag}=c(\rho)$, we obtain
\begin{equation*}
    \norm{\ukpd-\udag}^2\leq (1-\lambda_k)^2 c^2(\rho)+\lambda_k^2\left(1+\frac{L^2}{\kappa^2}\right)\rho^2+2\lambda_k(1-\lambda_k)\rho c(\rho)\leq c^2(\rho),
\end{equation*}
where the last inequality comes from \eqref{eq:crho} and the fact that 
\begin{equation*}
    (1-\lambda_k)^2 c^2(\rho)+\lambda_k^2\left(1+\frac{L^2}{\kappa^2}\right)\rho^2+2\lambda_k(1-\lambda_k)\rho c(\rho)\leq c^2(\rho)
\end{equation*}
if 
\begin{equation*}
    c(\rho)\geq \rho \frac{1-\lambda_k+\sqrt{1+\lambda_k(2-\lambda_k)\frac{L^2}{\kappa^{2}}}}{2-\lambda_k},
\end{equation*}
hence, in particular, for $c(\rho)$ as defined in \eqref{eq:crho}.
\end{proof}
\begin{remark}
It's worth highlighting that, thanks to Assumption \ref{as:average_mean_geometric} (b), the calculations for establishing the convergence and stability of GIRLI-GM \eqref{eq:GIRLI-GM} are exactly the same of \eqref{eq:GIRLI}, except for a factor $1/n$ which appears in some estimates for \eqref{eq:GIRLI}.
\end{remark}

The previous proposition provides a stopping criterion (the discrepancy principle) that is the iteration is stopped after $k^{\dag}=k^{\dag}(\delta,\yd)$ steps so that
\begin{equation}\label{eq:discr_princ1}
    \norm{\yd-F(u^{\delta}_{k^{\dag}})}\leq \tau\delta\leq \norm{\yd-F(\ukd)},\qquad 0\leq k<k^{\dag},
\end{equation}
where
\begin{equation}\label{eq:discr_princ2}
    \tau> 2(1-\lambda_{max})\frac{1+\eta}{E}.
\end{equation}
Next, we prove a result which gives information regarding the sum of the residuals till the iteration $k^{\dag}$, and in particular a sufficient condition which guarantees that $k^{\dag}$ is finite when $\delta>0$. The proof of the following result is completely analogous to the one in \cite{KalNeuSch08} but, for the reader's convenience, we report it here. 
\begin{corollary}\label{corollary}
    Let the Assumptions \ref{as:average_mean_geometric} (a) - \ref{as:fa} and the hypotheses of Proposition \ref{prop:suff_cond} be satisfied. The use of the discrepancy principle \eqref{eq:discr_princ1}-\eqref{eq:discr_princ1} gives that
    \begin{equation*}
        k^{\dagger}(\tau\delta)^2< \sum_{k=0}^{k^{\dagger}-1}\norm{\yd-F(\ukd)}^2\leq \frac{\rho^2}{D}\left[1+2\left(1+\frac{L^2}{\kappa^2}\right)\sum_{k=0}^{k^{\dag}-1}\lambda_k \right],
    \end{equation*}
    where
    \begin{equation*}
        D:=E-2(1-\lambda_k)(1+\eta)\tau^{-1}>0.
    \end{equation*}
\end{corollary}
\begin{proof}
Since $\lambda_k$ satisfies \eqref{eq:lambdak}, it is straightforward to prove from \eqref{eq:crho} that 
\begin{equation}\label{eq:aux3}
        \norm{\ukd-u^{\dag}}\leq c(\rho)<\rho\left(1+\frac{L^2}{\kappa^2}\right).
\end{equation}
Note that the last term in \eqref{eq:aux1} can be written equivalently as
\begin{equation}\label{eq:aux2}
\begin{aligned}
    \norm{F(\ukd)-\yd}^2&\left(\frac{2(1-\lambda_k)(1+\eta)\delta}{\norm{F(\ukd)-\yd}}-E\right)\\
    &\leq \norm{F(\ukd)-\yd}^2\left(2(1-\lambda_k)(1+\eta)\tau^{-1}-E\right)=D \norm{F(\ukd)-\yd}^2,
\end{aligned}
\end{equation}
where, in the first inequality, we have used the discrepancy principle, see \eqref{eq:discr_princ1}. Substituting \eqref{eq:aux2} in the last inequality of \eqref{eq:aux1}, we get that
\begin{equation}
\begin{aligned}
    &\norm{\ukpd-\udag}^2+D\norm{F(\ukd)-\yd}^2\\
    &\leq  (1-\lambda_k)^2\norm{(\ukd-\udag)}^2 + \frac{\lambda^2_k}{n^2}\left(1+\frac{L^2}{\kappa^2}\right)\norm{\sum_{i=1}^{n}(\udag-u^{(i)})}^2\\
        &+\frac{2\lambda_k}{n}(1-\lambda_k)\norm{\ukd-\udag}\norm{\sum_{i=1}^{n}(u^{(i)}-\udag)},
\end{aligned}        
\end{equation}
that is, from \eqref{eq:aux3},
\begin{equation}\label{eq:aux4}
\begin{aligned}
    &\norm{\ukpd-\udag}^2+D\norm{F(\ukd)-\yd}^2\\
    &\leq  (1-\lambda_k)^2\norm{(\ukd-\udag)}^2 + \lambda^2_k\left(1+\frac{L^2}{\kappa^2}\right)\rho^2+2\lambda_k(1-\lambda_k)\rho \norm{\ukd-\udag}\\
    &\leq  (1-\lambda_k)^2\norm{(\ukd-\udag)}^2 + \lambda^2_k\left(1+\frac{L^2}{\kappa^2}\right)\rho^2+2\lambda_k(1-\lambda_k)\rho^2\left(1+\frac{L^2}{\kappa^2}\right),
\end{aligned}        
\end{equation}
hence,
\begin{equation*}
\begin{aligned}
    \norm{\ukpd-\udag}^2+&D\norm{F(\ukd)-\yd}^2\\
    &\leq  (1-\lambda_k)^2\norm{(\ukd-\udag)}^2 +2\lambda_k\rho^2\left(1+\frac{L^2}{\kappa^2}\right)\\
    &\leq \norm{(\ukd-\udag)}^2 +2\lambda_k\rho^2\left(1+\frac{L^2}{\kappa^2}\right).
\end{aligned}
\end{equation*}
Therefore, we get
\begin{equation*}
    D\norm{F(\ukd)-\yd}^2\leq\norm{\ukd-\udag}^2-\norm{\ukpd-\udag}^2+2\lambda_k\rho^2\left(1+\frac{L^2}{\kappa^2}\right) 
\end{equation*}
and summing up, we find
\begin{equation*}
    D\sum_{k=0}^{k^{\dag}-1}\norm{F(\ukd)-\yd}^2\leq \norm{u_0-\udag}^2-\norm{u^{\delta}_{k^{\dag}}-\udag}^2+2\rho^2\left(1+\frac{L^2}{\kappa^2}\right)\sum_{k=0}^{k^{\dag}-1}\lambda_k, 
\end{equation*}
that is
\begin{equation*}
    \sum_{k=0}^{k^{\dag}-1}\norm{F(\ukd)-\yd}^2\leq\frac{\rho^2}{D}\left[1+2\left(1+\frac{L^2}{\kappa^2}\right)\sum_{k=0}^{k^{\dag}-1}\lambda_k\right]. 
\end{equation*}
If $\delta=0$, then we can choose $\tau$ sufficiently large such that $D\approx E$, hence
\begin{equation*}
    \sum_{k=0}^{\infty}\norm{F(\ukd)-\yd}^2\leq\frac{\rho^2}{D}\left[1+2\left(1+\frac{L^2}{\kappa^2}\right)\sum_{k=0}^{\infty}\lambda_k\right].
\end{equation*}
\end{proof}
We can now prove that the iterations of GIRLI, see \eqref{eq:GIRLI}, converges to a solution of the operator equation and, in particular, when $u_0=\frac{1}{n}\sum_{i=1}^{n}u^{(i)}$, if $\mathcal{N}(F'(u^*))\subset \mathcal{N}(F'(u))$, for all $u\in B_{\rho+c(\rho)}(u_0)$, where $u^*$ is the $u_0$-minimum norm solution and $\mathcal{N}(F'(\cdot))$ is the nullspace of the operator $F'$, then we get that $u_k\to u^*$. The following theorem follows the same line of reasoning contained in \cite{Sch98,KalNeuSch08}. 
\begin{theorem}\label{th:convergence}
  Let the Assumptions \ref{as:average_mean_geometric} (a) - \ref{as:fa} and the hypotheses of Proposition \ref{prop:suff_cond} be satisfied. Then, iterations \eqref{eq:GIRLI}, applied to the exact data $y$ (that is for $\delta=0$), converges to a solution $\udag$ of $F(u)=y$ in $B_{\rho+c(\rho)}(u_0)$. In particular, when $u_0=\frac{1}{n}\sum_{i=1}^{n}u^{(i)}$, if $\mathcal{N}(F'(u^*))\subset\mathcal{N}(F'(u))$, for all $u\in B_{\rho+c(\rho)}(u_0)$, where $u^*$ is the unique $u_0$-minimum norm solution, then $u_k\to u^*$, as $k\to+\infty$.    
\end{theorem}
\begin{proof}
Let us denote by $e_k:=u_k-\udag$.
Then, by means of Proposition \ref{prop:suff_cond}, there exist $M>0$ such that $\norm{e_k}< M$ and, consequently, a subsequence $e_{k_{n}}$ such that $\norm{e_{k_{n}}}\to \varepsilon\geq 0$, as $n\to\infty$.
From \eqref{eq:aux4}, specialized to the case where $\delta=0$, and using the fact that $\sqrt{1+L^2/\kappa^2}>1$, we find
\begin{equation*}
    \begin{aligned}
        \norm{u_{k+1}-\udag}^2&\leq\norm{\ukp-\udag}^2+D\norm{F(\uk)-y}^2\\
        &\leq (1-\lambda_k)^2\norm{\uk-\udag}^2+4\lambda_k^2\left(1+\frac{L^2}{\kappa^2}\right)\rho^2\\
        &\hspace{0.5cm}+2\lambda_k(1-\lambda_k)\rho\sqrt{1+\frac{L^2}{\kappa^2}}\norm{u_k-\udag}\\
        &=\left((1-\lambda_k)\norm{\uk-\udag}+2\lambda_k\sqrt{1+\frac{L^2}{\kappa^2}}\rho\right)^2,
    \end{aligned}
\end{equation*}
hence
\begin{equation}\label{eq:aux5}
        \norm{e_{k+1}}\leq (1-\lambda_k)\norm{e_k}+2\lambda_k\sqrt{1+\frac{L^2}{\kappa^2}}\rho.
\end{equation}
Applying recursively \eqref{eq:aux5}, we find
\begin{equation*}
    \norm{e_{k+1}}\leq \prod_{j=m}^{k}(1-\lambda_j)\norm{e_m}+2\rho\sqrt{1+\frac{L^2}{\kappa^2}}\sum_{j=m}^{k}\prod_{h=j+1}^{k}(1-\lambda_h),
\end{equation*}
with the convention that $\prod_{h=k+1}^{k}(1-\lambda_h)=1$. From a straightfoward calculation, since
\begin{equation*}
    \sum_{j=m}^{k}\prod_{h=j+1}^{k}(1-\lambda_h)=1-\prod_{j=m}^{k}(1-\lambda_j),
\end{equation*}
we find that 
\begin{equation*}
\norm{e_{k+1}}\leq \norm{e_m}\prod_{j=m}^{k}(1-\lambda_j)+2\rho\sqrt{1+\frac{L^2}{\kappa^2}}\left(1-\prod_{j=m}^{k}(1-\lambda_j) \right).
\end{equation*}
Without loss of generality, we assume that $k_n<k+1<k_{n+1}$, then, from the previous inequality,
\begin{equation*}
    \norm{e_{k+1}}\leq \norm{e_{k_{n}}}+\left(2\rho\sqrt{1+\frac{L^2}{\kappa^2}}-\norm{e_{k_{n}}}\right)\left(1-\prod_{j=k_n}^{k}(1-\lambda_j) \right),
\end{equation*}
and
\begin{equation*}
    \norm{e_{k_{n+1}}}\leq \norm{e_{k+1}}+\left(2\rho\sqrt{1+\frac{L^2}{\kappa^2}}-\norm{e_{k+1}}\right)\left(1-\prod_{j=k+1}^{k_{n+1}}(1-\lambda_j) \right),
\end{equation*}
that is
\begin{equation*}
    \norm{e_{{k+1}}}\geq \norm{e_{k_{n+1}}}-\left(2\rho\sqrt{1+\frac{L^2}{\kappa^2}}-\norm{e_{k+1}}\right)\left(1-\prod_{j=k+1}^{k_{n+1}}(1-\lambda_j) \right).
\end{equation*}
Therefore, we find that $\lim_{k\to+\infty}\norm{e_k}=\varepsilon$, since the a priori assumption $\sum_{k=0}^{\infty}\lambda_k<+\infty$ implies that 
\begin{equation}\label{eq:aux6}
    \prod_{j=m}^{k}(1-\lambda_k)\to 1,\qquad \textrm{for}\ m\to+\infty, k>m.
\end{equation}
Now, we can prove the convergence of the iteration to the solution of minimum norm with respect to $u_0$. If, for $n\geq k$, we choose $m$ such that $n\geq m\geq k$, and
\begin{equation*}
    \norm{F(u_m)-y}\leq \norm{F(u_i)-y},\qquad \textrm{for all}\ i\ \textrm{such that}\ k\leq i\leq n, 
\end{equation*}
then the convergence of $u_k$ is guaranteed showing that $e_k$ is a Cauchy sequence, namely that $\langle e_m-e_n,e_m\rangle$ and $\langle e_m-e_k,e_m\rangle$ tend to zero as $k\to+\infty$. We show the result for $\langle e_m-e_n,e_m\rangle$. The same hold for the other term. 

From Corollary \ref{corollary}, when $\delta=0$, we deduce that $\norm{y-F(u_k)}\to 0$ as $k\to+\infty$, that is $u_k$ converges to a solution of the operator equation $F(u)=y$.
From the definition of the iteration, see \eqref{eq:GIRLI}, and recalling that $u_0=\frac{1}{n}\sum_{i=1}^{n}u^{(i)}$, we find, by adding and subtracting suitable quantities,
\begin{equation*}
    e_{k+1}=(1-\lambda_k)e_k-F'(u_k)^*(F(u_k)-y)-\lambda_k(u^*-u_0),
\end{equation*}
hence, recursively,
\begin{equation*}
\begin{aligned}
    e_{k+1}&=\prod_{j=m}^k(1-\lambda_j)e_m-\sum_{j=m}^kF'(u_j)^*(F(u_j)-y)\prod_{h=j+1}^k(1-\lambda_h)\\
    &-(u^*-u_0)\left(1-\prod_{j=m}^{k}(1-\lambda_j)\right),\qquad m\leq k.
\end{aligned}
\end{equation*}
Thanks to this equality, we obtain that
\begin{equation*}
\begin{aligned}
    \big|\langle e_m-e_n,e_m\rangle\big|&\leq \left(1-\prod_{j=m+1}^{n}(1-\lambda_j)\right)\big|\langle u_{m+1}-u_0,e_m\rangle\big|\\
    &+\sum_{j=m+1}^n\prod_{h=j+1}^n (1-\lambda_h) \Big|\langle F(u_j)-y, F'(u_j)(u_{m+1}-u^*)\rangle\Big|.
\end{aligned}
\end{equation*}
Since $\langle u_{m+1}-u_0,e_m\rangle$ is bounded and \eqref{eq:aux6} holds, we have that the first term on the right-hand side of the previous inequality tends to zero as $m\to\infty$. The second term can be estimated using the tangential cone condition \eqref{eq:tcc}, where $r=\rho+c(\rho)$, and the fact that $1-\lambda_j<1$, for all $j$, hence
\begin{equation*}
    \sum_{j=m+1}^n\prod_{h=j+1}^n (1-\lambda_h)\Big|\langle F(u_j)-y, F'(u_j)(u_{m+1}-u^*)\rangle\Big|\leq (1+3\eta)\sum_{j=m+1}^{n}\norm{y-F(u_j)}^2,
\end{equation*}
hence tends to zero as $m\to \infty$. Therefore, both $e_{k+1}$ and $u_{k+1}$ are Cauchy sequences, hence $u_{k} \to \udag$, as $k\to+\infty$.\\
Now, since $\mathcal{N}(F'(u^*))\subset\mathcal{N}(F'(u))$, for all $u\in B_{\rho+c(\rho)}(u_0)$ then 
\begin{equation*}
    u_{k+1}-u_k\in \mathcal{R}(F'(u_k)^*)\subset \mathcal{N}(F'(u_k))^{\perp}\subset \mathcal{N}(F'(u^*))^{\perp},
\end{equation*}
hence $u_k-u_0\in\mathcal{N}(F'(u^*))^{\perp}$, for all $k\in\mathbb{N}$. 
This implies that 
\begin{equation*}
    u^*-\udag=u^*-u_0+u_0-\udag\in\mathcal{N}(F'(u^*))^{\perp}.
\end{equation*}
Suppose that $\udag\neq u^*$ then,
\begin{equation*}
\begin{aligned}
    \norm{F'(u^*)(\udag-u^*)}&\leq \norm{F(\udag)-F(u^*)-F'(u^*)(\udag-u^*)}+\norm{F(u^*)-F(\udag)}\\
    &\leq (1+\eta) \norm{F(u^*)-F(\udag)}=0,
\end{aligned}
\end{equation*}
hence $\udag-u^*\in \mathcal{N}(F'(u^*))\cap \mathcal{N}(F'(u^*))^{\perp}=\{0\}$.
\end{proof}
Analogously to the iteratively regularized Landweber iteration, the convergence of the generalized iteratively regularized Landweber iteration cannot be guaranteed in the context of noisy data, due to the possibility that $\yd$ may not lie within the range of the operator $F$. Specifically, a stable approximation of a solution to $F(u) = y$ is only achievable when employing the discrepancy principle \eqref{eq:discr_princ1}-\eqref{eq:discr_princ2}. This entails terminating the iteration after a finite number of steps. In the next theorem, making use of the discrepancy principle, we prove that the generalized iteratively regularized Landweber iteration is stable under perturbation of the data, that is it acts as a regularization method. 
\begin{theorem}
  Let the Assumptions \ref{as:average_mean_geometric} (a) - \ref{as:fa} and the hypotheses of Proposition \ref{prop:suff_cond} be satisfied. Let us denote by $k^*=k^*(\delta,\yd)$ a positive integer that satisfies the discrepancy principle \eqref{eq:discr_princ1}-\eqref{eq:discr_princ2}. Then the Generalized Iteratively Regularized Landweber iterates $u^{\delta}_{k^*}$ converge to a solution of $F(u)=y$ as $\delta\to 0^+$. If $\mathcal{N}(F'(\udag))\subset \mathcal{N}(F'(u))$, for all $u\in B_{\rho+c(\rho)}(u_0)$, then $u^{\delta}_{k^*}$ converges to $u^*$ as $\delta\to 0^+$.
\end{theorem}
\begin{proof}
    Using the same notation above, we denote by $\udag$ the limit of the generalized iteratively regularized Landweber iteration when the precise data $y$ are employed. Let, moreover, $\delta_n$ be a sequence that converges to zero as $n\to \infty$ and denote by $y_n:=y^{\delta_n}$ the perturbed data and with $k_n=k^*(\delta_n,y_n)$ the stopping index provided by the application of the discrepancy principle to the pair $(\delta_n,y_n)$. 
We now distinguish two cases as $n\to\infty$:
    \begin{enumerate}
	\item $k_n\to k,\, k\in \mathbb{R}^+$;
	\item $k_n \to +\infty$.
	\end{enumerate}
\textit{Case (i)}: to avoid technicalities, we can assume that $k_n=k$, for all $n\in\mathbb{N}$. Therefore, from \eqref{eq:discr_princ1} we get
	\begin{equation*}
	\norm{F(u_{k}^{\delta_n})-y_n}\leq \tau \delta_n.
	\end{equation*}	 
	From the continuity hypotheses on $F$ and $F'$ we have that
	\begin{equation*}
	u^{\delta_n}_k \to u_k,\qquad F(u^{\delta_n}_k)\to F(u_k)=y,\qquad \text{as}\ \, n\to \infty.
	\end{equation*}
Therefore, we find that $u_k$ is a solution of the operator equation $F(u)=y$ and thus $\udag=u_k$.
\textit{Case (ii)}: We choose $k$ sufficiently large such that $k_n> k$. By means of \eqref{eq:aux5}, we find that
\begin{equation}\label{eq:aux7}
\norm{u^{\delta_n}_{k_n}-\udag}\leq \prod_{j=k}^{k_n-1}(1-\lambda_j)\norm{u^{\delta_n}_k-\udag}+4\rho\sqrt{1+\frac{L^2}{\kappa^2}}\left(1-\prod_{j=k_n-1-k}^{k_n-1}(1-\lambda_j) \right).
\end{equation}
We note that
\begin{equation*}
    \norm{u^{\delta_n}_k-\udag}\leq \norm{u^{\delta_n}_k-u_k}+\norm{u_k-\udag},
\end{equation*}
hence, given $\varepsilon>0$, from Theorem \ref{th:convergence}, we can find  $k=k(\varepsilon)$ such that $\norm{u_k-\udag}\leq \frac{\varepsilon}{2}$. Choosing $n$ sufficiently large, we also have that $\norm{u^{\delta_n}_k-u_k}<\frac{\varepsilon}{2}$, for all $n>n(\varepsilon,k)$. Then, using these two estimates, the fact that the second term on the right-hand side of \eqref{eq:aux7}  is going to zero thanks to \eqref{eq:aux6}, and that $1-\lambda_j<1$, for all $j=k,\ldots,k_n-1$, we finally get that $u^{\delta_n}_{k_n}\to \udag$, as $n\to\infty$. The second statement of the theorem follows using the same line of reasoning and recalling that if $\mathcal{N}(F'(\udag))\subset \mathcal{N}(F'(u))$, for all $u\in B_{\rho+c(\rho)}(u_0)$, then, from Theorem \ref{th:convergence}, the iterates are converging to $u^*$.   
\end{proof}
\begin{remark}
    The same results hold also in the case of the iterates \eqref{eq:GIRLI-GM}, since they follow from the same estimates provided in Lemma \ref{prop:suff_cond}.
\end{remark}

\section{Numerical experiments}\label{sec:NM}
We present some numerical results comparing various iterative methods. Specifically, in the first section, we will focus the attention on the classical Landweber iteration, the iteratively regularized Landweber method, as presented in \cite{Sch98}, the data-driven iteratively regularized Landweber iteration, as proposed in \cite{AspBanOktSch20}, and finally the generalized iteratively regularized Landweber method analyzed in this paper.
In the second section, instead, we compare the generalized iteratively regularized Landweber iteration with the outcomes given by GIRLI-GM and a variant of the iteratively regularized Landweber iteration.\\ 
We focus the attention on linear problems, taking as prototype of the operator $F$ the Radon transform, that is
\begin{equation}\label{eq:operator_R}
Ru=y.
\end{equation}
We refer, for instance, the reader to \cite{Kuc14,OlafQuin06} for the definition and properties of the Radon transform.\\     
We present a concise summary in the table below, see Table \ref{tab: table0}, outlining the methods slated for comparison appearing in the subsequent two sections, accompanied by their respective acronyms and iterates tailored specifically for the Radon transform scenario.
{	\begin{table}[!h]\caption{Summary of the methods and their acronyms.}\label{tab: table0}
		\begin{center}
			{\renewcommand{\arraystretch}{1.2}
				\begin{tabular}{|l|c|}
					\cline{1-2} & \vspace{-0.2cm} \\
					\multirow{2}{4em}{Acronym} &  Name of the method \\ & and \\
                    & Expression of Iterates \\ \hline
                    & \\
                    \multirow{2}{4em}{GIRLI} & Generalized Iteratively Regularized Landweber Iteration \\
                    & \vspace{-0.3cm} \\
                    & $\ukpd := (1-\lambda_k)\ukd - \omega_R R^*\bigl( R\ukd-\yd \bigr) + \frac{\lambda_k}{n}\sum_{i=1}^{n} u^{(i)}$  \\ 
                    & \\ \cline{1-2}
                    & \\
                    {GIRLI-adapt} & Generalized Iteratively Regularized Landweber Iteration \\
                    & with adaptation of the stabilizing term \\
                    & \vspace{-0.3cm} \\
                    & $\ukpd := (1-\lambda_k)\ukd - \omega_R R^*\bigl( R\ukd-\yd \bigr) + \frac{\lambda_k}{n_k}\sum_{i=1}^{n_k} u^{(i)}$,  \\
                    & \vspace{-0.3cm} \\
                    & $\norm{\ukd-u^{(i)}}\leq \textrm{tol}$, 
                    $i=1,\ldots,n_k$, $k>k_0$\\
                    & \\ \cline{1-2}
                    & \\
                     {GIRLI-GM} & Generalized Iteratively Regularized Landweber Iteration \\
                    & with Geometric Mean \\
                    & \vspace{-0.3cm} \\
                    &  $\ukpd := (1-\lambda_k)\ukd - \omega_R R^*\bigl( R\ukd-\yd \bigr) +\lambda_k\sqrt[n]{\prod_{i=1}^{n}u^{(i)}}$\\ 
                    & \\ \cline{1-2}
                    & \\
					\multirow{2}{4em}{DDIRLI} & Data-Driven Iteratively Regularized Landweber Iteration  \\
                    & \vspace{-0.3cm} \\
                    & $\ukpd := \ukd - \omega_R R^*\bigl( R\ukd-\yd \bigr) - {\beta^{\delta}_k}{A}^*\bigl({A}\ukd -\yd \bigr)$  \\ 
                    & \\
                    \cline{1-2}
                    & \\
                   \multirow{2}{4em}{IRLI} & Iteratively Regularized 
                    Landweber Iteration \\
                    & \vspace{-0.3cm} \\
                    & $\ukpd := \ukd - \omega_R R^*\bigl( R\ukd-\yd \bigr) - \lambda_k (\ukd -u^{(0)})$  \\ 
                    & \\ 
                    \cline{1-2}
                    & \\
                    {IRLI-revised} & Iteratively Regularized Landweber Iteration Revised \\
                    & \vspace{-0.3cm} \\
                    &  $\ukpd := \ukd - \omega_R R^*\bigl( R\ukd-\yd \bigr) -\mu_k (\ukd - u^{(i)}),\qquad i\equiv k\mod n$\\ 
                    & \\ 
                    \cline{1-2}
                    & \\
					LANDWEBER & Landweber method \\
                    &  \vspace{-0.3cm} \\
                    & $\ukpd := \ukd - \omega_R R^*\bigl( R\ukd-\yd \bigr)$\\ 
                    & \\
                    \hline   
				\end{tabular}
			}
		\end{center}   
	\end{table} 
}

\subsubsection*{GIRLI vs. classical and specific data-driven methods}
The iterates in \eqref{eq:GIRLI} specialized to the case of the Radon transform reads as 
\begin{equation}\label{eq:GIRLI-Radon}
\begin{aligned}
&\ukpd := (1-\lambda_k)\ukd - \omega_R R^*\bigl( R\ukd-\yd \bigr) + \frac{\lambda_k}{n}\sum_{i=1}^{n} u^{(i)},\qquad k\in\mathbb{N},\\
\end{aligned}
\end{equation}
where $R^*$ is the backprojection operator and $\omega_R$ is a positive constant such that $\norm{R^*(R\ukd-\yd)}\leq1$.
Regarding the data-driven approach proposed in \cite{AspBanOktSch20}, we recall that 
\begin{equation}\label{eq:DDIRLI-Radon}
\ukpd := \ukd - \omega_R R^*\bigl( R\ukd-\yd \bigr) - {\beta^{\delta}_k}{A}^*\bigl({A}\ukd -\yd \bigr), \qquad k\in\N,
\end{equation} 
where ${A}$ is an handcrafted operator and $\beta^{\delta}_k\leq C\norm{R\ukd-\yd}^2$, with $C>0$. We refer the reader to Section 3 of \cite{AspBanOktSch20} for more details. Here, we recall only the necessary tools for the implementation (see below). 
Finally, we recall that the results of \eqref{eq:GIRLI-Radon} are also compared with the outcomes of the Landweber iteration, that is
\begin{equation}\label{eq:Land_R}
\ukpd := \ukd - \omega_R R^*\bigl( R\ukd-\yd \bigr), \qquad k\in\N,
\end{equation}
and of the Iteratively Regularized Landweber Iteration given by
\begin{equation}\label{eq:IRLI_classic}
\ukpd := \ukd - \omega_R R^*\bigl( R\ukd-\yd \bigr) - \lambda_k (\ukd -u^{(0)}), \qquad k\in\N.
\end{equation} 
For our purposes, we use the following set of parameters and training data: 
\begin{enumerate}
    \item\label{item: point1}We use a set of $180$ equally spaced angles $\theta$ distributed within the interval $[0, \pi)$, that is we take $\Theta=\Big\{0,\frac{\pi}{180},\frac{\pi}{90},\cdots,\frac{179\pi}{180}\Big\}$;
    \item \label{item: point2}A dataset comprising $n=150$ grayscale images (denoted by $u^{(i)}$) of handwritten digits sourced from the MNIST database, along with their corresponding sinograms (denoted by ${y}^{(i)}$) generated using the 180 directions specified in \eqref{item: point1}. 
\end{enumerate}
GIRLI and DDIRLI use the following data pairs (or only a part) for their implementation  
\begin{equation}\label{eq: train_data_finite}
({u}^{(i)},{y}^{(i)})\in \mathbb{R}^N \times \mathbb{R}^M,\,\, \textrm{for}\,\, i=1,\cdots,n,
\end{equation}
where $N$, $M$ are positive integers.  
In particular, the matrix $A\in\mathbb{R}^{M\times N}$ is defined as the linear map such that 
\begin{equation}\label{eq: def_A}
AU=Y, 
\end{equation}
where $U\in\mathbb{R}^{N\times n}$ and $Y\in\mathbb{R}^{M\times n}$ are the matrices which contain columnwise the data $u^{(i)}$ and $y^{(i)}$, respectively, for $i=1,\cdots,n$. 
In \cite{AspBanOktSch20}, the matrix $A$ is obtained by utilizing a Singular Value Decomposition (SVD) on $U$, that is 
\begin{equation}\label{eq: def_A_calc}
A=YU^{\dagger}.
\end{equation}
In the implementation, for the set of parameters we make the following choices:
\begin{enumerate}
	\item[(a)] We choose
	\begin{equation}\label{value lambdak}
	\omega_R=10^{-2}\quad \textrm{and}\quad C= 77\times10^{-6} (\textrm{in DDIRLI}).
	\end{equation}
	\item[(b)] \label{stopping_criteria} We employ the discrepancy principle \eqref{eq:discr_princ1} as the stopping rule, with the specific choice of $\tau$ detailed for each test (refer to tables below), or a maximum iteration limit of 1000. 
	\item[(c)] The synthetic data ${\yd}$ is created by introducing Gaussian-distributed noise with a mean of zero and a variance of $\sigma^2$, which is defined for each test, to the matrix of the exact data ${y}$.
	\item[(d)] In the iterations of IRLI and Landweber, we initialize with the guess $u_0$ selecting a digit similar to the one we intend to reconstruct, which is chosen among the digits in the MNIST validation set (see for an example Figure \ref{fig: choice_u0}). Additionally, in IRLI, we enforce that $u^{(0)}=u_0$.
    In DDIRLI, we adopt the assumption that $u_0=0$ to facilitate the guidance of iterations by the supplementary damping term ${\beta^{\delta}_k} {A}^*\bigl({A}\uk -\yd \bigr)$. In GIRLI, as stated in the theoretical part, we use as initial guess $u_0=\frac{1}{n}\sum_{i=1}^{n}u^{(i)}$.  
\end{enumerate}
\begin{figure}[!h]
\centering
		\includegraphics[scale=0.6]{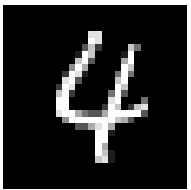}\hspace{2cm}
		\includegraphics[scale=0.6]{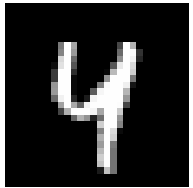}\hspace{2cm}
	\caption{Left-hand side: image to be reconstructed. Right-hand side: initial guess in IRLI and Landweber iterations.}\label{fig: choice_u0}
\end{figure}
We utilized the training and validation sets from the MNIST dataset of handwritten digits \cite{LecunBottBengHaff}. Our study comprises four distinct numerical experiments. In \textbf{Test 1}, we reconstruct a digit from the validation set of MNIST, specifically an image not employed in creating the matrix $A$ or included in $u_0$ for both IRLI and GIRLI. In \textbf{Test 2}, we consider the same image as in Test 1, but we modify the iteration of GIRLI in the following way: starting from a certain step $k$, chosen arbitrarily, the elements $u^{(i)}$ contained in $u_0$ (in GIRLI) are compared with $u_k$, at each iteration. If the norm of the difference $u^{(i)} - u_k$ is sufficiently large (e.g., greater than 3.2 in the following numerical examples), then the element $u^{(i)}$ is removed from the set of training pairs.
Finally, in \textbf{Test 3} and in \textbf{Test 4} we consider the same situation of Test 1 and Test 2 but assuming only a knowledge of limited data, that is we do not know the entire sinogram $\yd$ but only a portion of it.

The choices and values of key parameters utilized for data generation and the implementation of various schemes are meticulously outlined in the tables below (Tables \ref{tab: table1}, \ref{tab: table2}, \ref{tab: table3}, \ref{tab: table4}). In the left section of each table, we specify the variance choice in Gaussian noise, along with the values of $\delta$ and $\tau$. The right section presents select numerical test results, encompassing the total number of iterations before meeting one of the two stopping criteria in (b), the computational time, and, finally, the relative error $\norm{u_{\textrm{true}}-u_{\textrm{rec}}}^2_2/\norm{u_{\textrm{true}}}^2_2$, where $\norm{\cdot}_2$ is the Euclidean norm. Here, $u_{\textrm{true}}$ is the image to be reconstructed, and $u_{\textrm{rec}}$ is the reconstructed image derived from various iterative methods. Importantly, note that the execution time of DDIRLI is influenced by both the construction of matrix $A$ and the iterations of \eqref{eq:DDIRLI-Radon}.

\textbf{Test 1.} (Figure \ref{fig: test1}). \textit{Target: reconstruct a digit from the validation set}. 
The initial data $\yd$ is obtained from the sinogram of the true image $y$ by introducing Gaussian-distributed noise with a mean of zero and variance $\sigma^2=0.5$. For the discrepancy principle \eqref{eq:discr_princ1}, we select $\tau=1.1$ (refer to Table \ref{tab: table1}). Furthermore, we set a maximum allowed number of iterations equal to 1000. GIRLI and IRLI incur a significant computational time cost, requiring a high number of iterations, with both algorithms often reaching the maximum allowable iterations. Despite this, the outcomes are commendable, albeit the DDIRLI algorithm stands out for delivering the best results both in terms of the computational time cost and the relative error. Note that the computational time cost of DDIRLI comprises both the time spent on DDIRLI iterations and the time required to construct the matrix A using the Singular Value Decomposition method.  

{	\begin{table}[!h]\caption{Test 1. Left part: Parameters used in the test. Right part: some of the results of the test.}\label{tab: table1}
		\begin{center}
			{\renewcommand{\arraystretch}{1.2}
				\begin{tabular}{|l|c|c|c||c|c|c|}
					\cline{1-7} & & & & & & \vspace{-0.2cm} \\
					Method & $\sigma^2$-noise& $\delta$ & $\tau$ & Iterations & Comp. Time (s) & $\frac{\norm{u_{\textrm{true}} - u_{\textrm{rec}}}_2}{\norm{u_{\textrm{true}}}_2}$\\ \hline
                    {GIRLI} &  &  &  & 999 & 5.4259 & 0.2355 \\ \cline{1-1}\cline{5-7}
					{DDIRLI} &  &  &  & 35 & 0.6123 & 0.2089 \\ \cline{1-1}\cline{5-7}
					{IRLI} & 0.5  & 13.6477
					& 1.1  & 999 & 5.6677 & 0.2904 \\ \cline{1-1}\cline{5-7} 
					LANDWEBER &  & &  & 98 &0.5445
					& 0.2397  \\ \hline   
				\end{tabular}
			}
		\end{center}   
	\end{table} 
}

\begin{figure}[!h]
	\centering
        \includegraphics[scale=0.35]{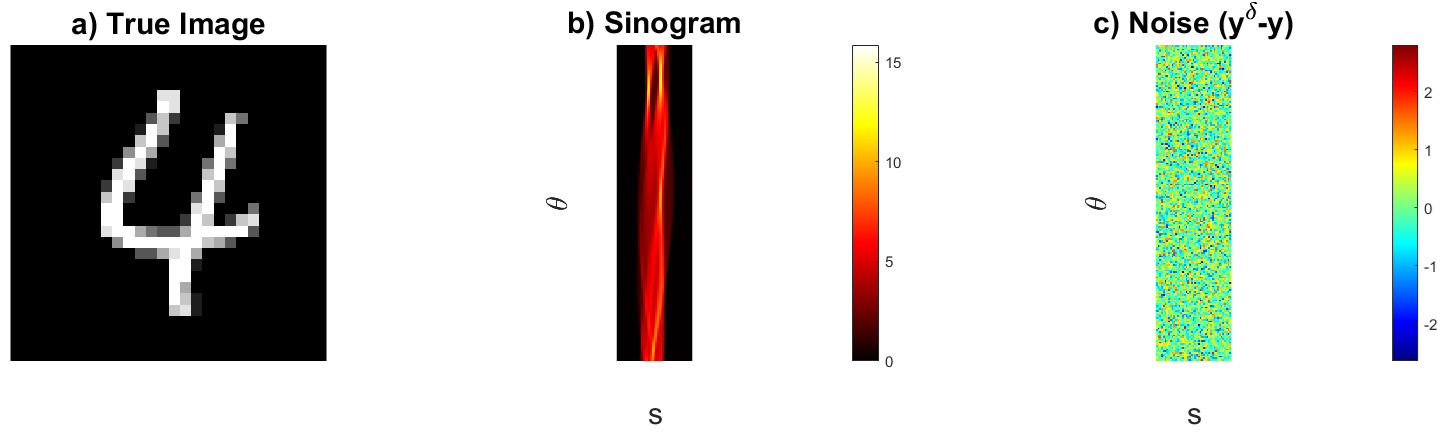}
	\includegraphics[scale=0.35]{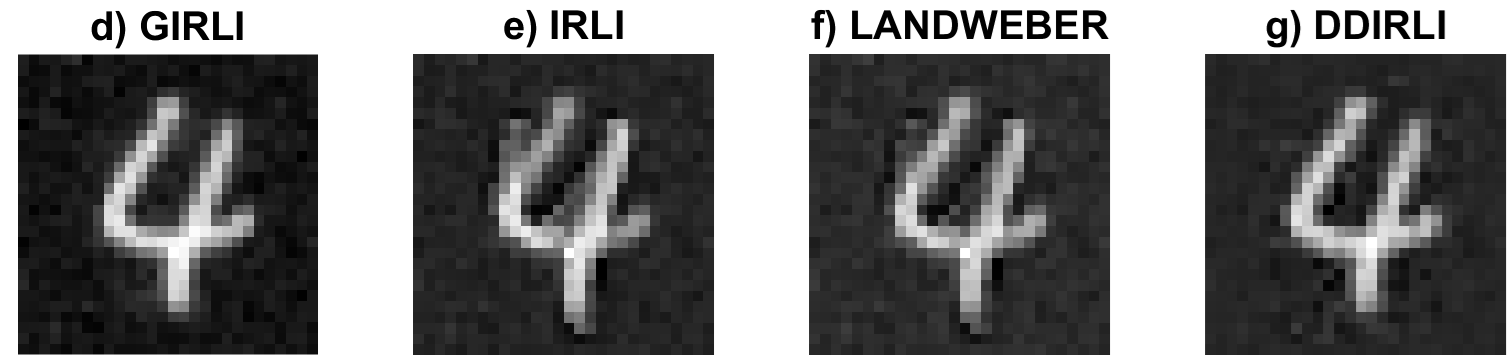}
	\caption{Test 1. Reconstruction of an image from the validation set. a) Image to be reconstructed; b) Plot of the cropped sinogram of the true image, i.e., $y$; c) Plot of the Gaussian noise, i.e., $\mathbf{y_d} - \mathbf{y}$; d)-g) Reconstructions using various methods.}\label{fig: test1}
\end{figure}

\textbf{Test 2}. (Figure \ref{fig: test2}). \textit{Target: reconstruct a digit from the validation set, removing from the term $u_0$ in GIRLI, after a certain number of iterations, the images that differ excessively from those being obtained through iterates.} We consider the same setting of Test 1 taking into account a modification of the GIRLI procedure, eliminating from $\sum_{i=1}^{n}(u_k-u^{(i)})$, after a certain number of iterations, fixed a priori, the pictures which are sufficiently different from the results obtained by the iteration procedure. We call this new procedure GIRLI-adapt.\\
With this method, we prioritize utilizing images identified by the algorithm as closely resembling our intended reconstruction. At each iteration, we compute the difference between the values of the iteration $u_{k+1}$ and those of the images $u^{(i)}$. We compute the Euclidean norm of the computed difference, and if the result surpasses a predefined threshold, the associated image is excluded from $\sum_{i=1}^{n}(u_k-u^{(i)})$. In Figure \ref{fig: test2_adapt} this algorithm has been implemented: after the tenth iteration, the described comparison
was made, with a tolerance set for the norm at $3.2$, eliminating from $\sum_{i=1}^{n}(u_k-u^{(i)})$ the images that do not meet the tolerance criterion, namely for which the norm is greater or equal than $3.2$. The only images that are retained are the four appearing in Figure \ref{fig: test2_adapt}. One can observe a small improvement of the result of GIRLI-adapt with respect to GIRLI, see Figure \ref{fig: test2}.

\begin{figure}[!h]
	\centering
        {\includegraphics[scale=0.35]{images/to_be_reconstructed.png}}\hspace{2cm}
	{\includegraphics[scale=0.35]{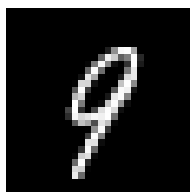}}\hspace{0.5cm}
	{\includegraphics[scale=0.37]{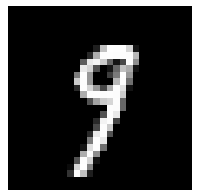}}\hspace{0.5cm}
        {\includegraphics[scale=0.37]{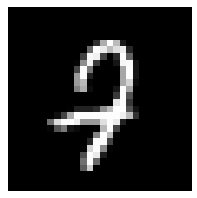}}\hspace{0.5cm}
          {\includegraphics[scale=0.37]{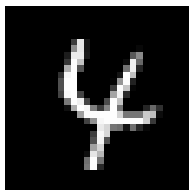}}\vspace{0.5cm}\\
          \includegraphics[scale=0.35]{images/Test2/digit_four.png}\hspace{0.5cm}
	{\includegraphics[scale=0.37]{images/Test2/digit_nine.png}}
	\caption{Test 2. Reconstruction of a digit from the validation set, removing from the term $u_0$ in GIRLI, after a certain number of iterations, the images that differ excessively from those being obtained through iterates. First line (from the left to the right): image to be reconstructed; images recognized by the system as relevant for reconstruction in the early steps. Second line: the only two images retained by the system until the final iteration.}\label{fig: test2_adapt}
\end{figure}

{	\begin{table}[!h]\caption{Test 2. Left part: Parameters used in the test. Right part: some of the results of the test.}\label{tab: table2}
		\begin{center}
			{\renewcommand{\arraystretch}{1.2}
				\begin{tabular}{|l|c|c|c||c|c|c|}
					\cline{1-7} & & & & & & \vspace{-0.2cm} \\
					Method & $\sigma^2$-noise& $\delta$ & $\tau$ & Iterations & Comp. Time (s) & $\frac{\norm{u_{\textrm{true}} - u_{\textrm{rec}}}_2}{\norm{u_{\textrm{true}}}_2}$\\ \hline  
				{GIRLI-adapt} & 0.5 & 13.6477 & 1.1 & 105 & 1.1511 & 0.1937 \\ \hline 
				\end{tabular}
			}
		\end{center}   
	\end{table} 
}

\begin{figure}[!h]
	\centering
	\includegraphics[scale=0.35]{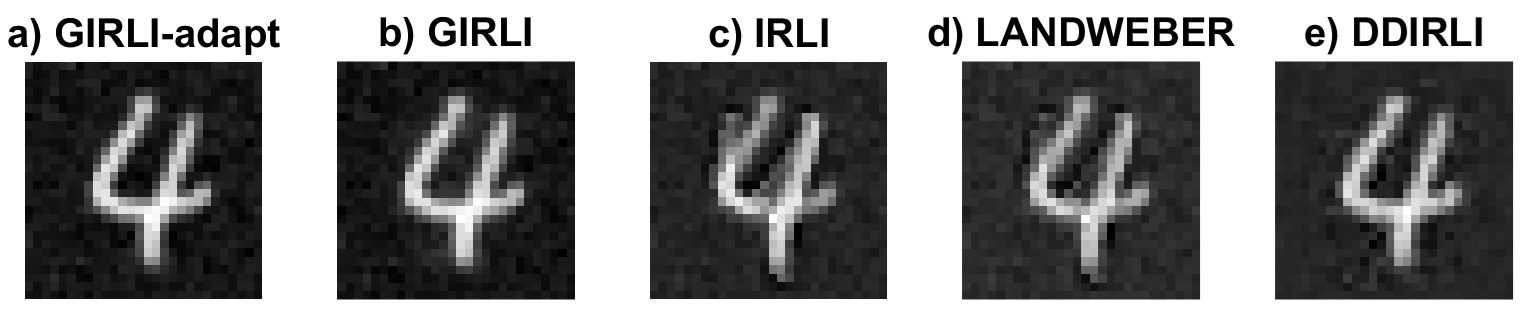}
	\caption{Test 2. Image reconstruction from the validation set. a)-e) Reconstructions using different methods. The adapted version of GIRLI, known as GIRLI-adapt, and DDIRLI produce highly favorable outcomes. However, GIRLI-adapt achieves a similar result to DDIRLI but with significantly more iterations and execution time. It is noteworthy that the results of IRLI and Landweber closely resemble the initially chosen guess.}\label{fig: test2}
\end{figure}

\textbf{Test 3} (Figure \ref{fig: test3}). {\textit{Target: reconstruct a digit from the validation set using partial data}}.
We explore the scenario of partial noisy measurements, achieved by extracting a truncated version of the sinogram $\yd$ with Gaussian-distributed noise of zero mean and variance $\sigma=0.03$. The cropped sinogram retains information only in specific directions, with all other directions set to zero. See Figure \ref{fig: test3}. 
GIRLI and DDIRLI provide good results. As expected, GIRLI is significantly slower compared to DDIRLI, which provides excellent reconstructions in a short amount of time.  
\\

{	\begin{table}[!h]\caption{Test 3. Left part: Parameters used in the test. Right part: some of the results of the test.}\label{tab: table3}
		\begin{center}
			{\renewcommand{\arraystretch}{1.2}
				\begin{tabular}{|l|c|c|c||c|c|c|}
					\cline{1-7} & & & & & & \vspace{-0.2cm} \\
					Method & $\sigma^2$-noise& $\delta$ & $\tau$ & Iterations & Comp. Time (s) & $\frac{\norm{u_{\textrm{true}} - u_{\textrm{rec}}}_2}{\norm{u_{\textrm{true}}}_2}$\\ \hline  
                    {GIRLI} &  & &  & 999 & 7.6294 & 0.4853 \\
                        \cline{1-1}\cline{5-7}
                    {DDIRLI} & 0.03 & 2.78 & 5 & 31 & 0.7669 & 0,3518 \\
                    \cline{1-1}\cline{5-7}
				{IRLI} &   &  &  & 999 & 8.2140 & 0.4565 \\ \cline{1-1}\cline{5-7} 
					LANDWEBER & & &  & 85 & 0.6486
					& 0.4516 \\ \hline 
				\end{tabular}
			}
		\end{center}   
	\end{table} 
}

\begin{figure}[!h]
	\centering
        \includegraphics[scale=0.35]{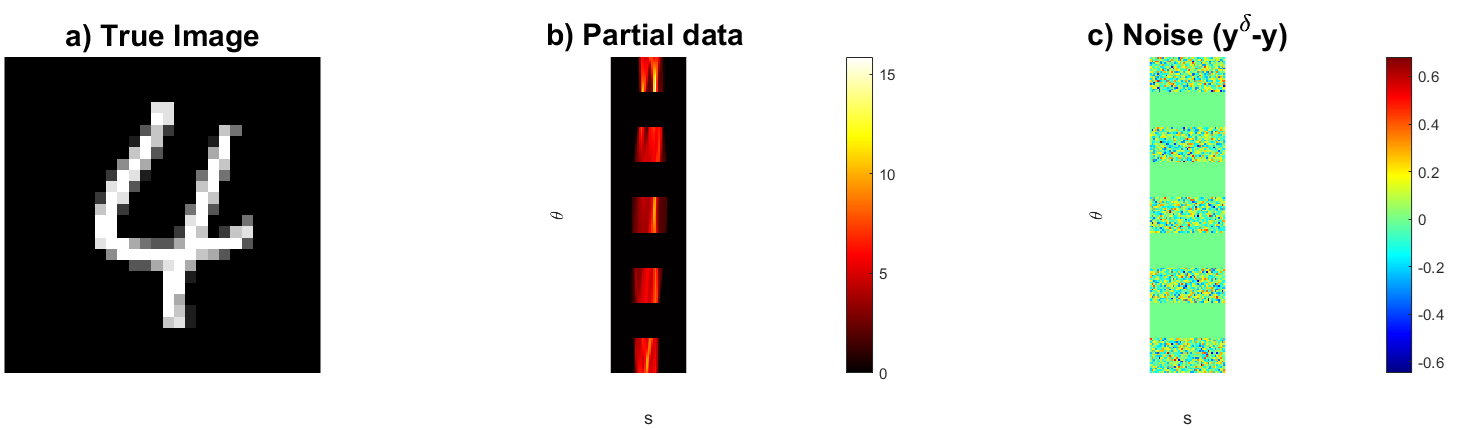}
	\includegraphics[scale=0.35]{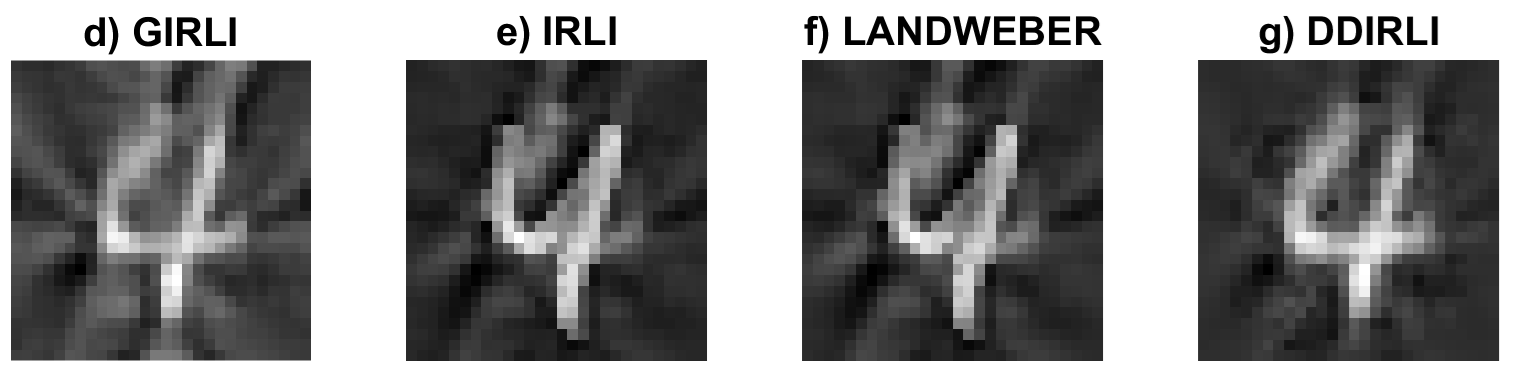}
	\caption{Test 3. Reconstruction of an image from the validation set using partial data. a) Image to be reconstructed; b) Plot of the cropped sinogram of the true image, i.e., $y$; c) Plot of the Gaussian noise, i.e., $\mathbf{y_d} - \mathbf{y}$; d)-g) Reconstructions using various methods.}\label{fig: test3}
\end{figure}

\textbf{Test 4} (Figure \ref{fig: test4}). {\textit{Target: reconstruct a digit from the validation set using partial data, removing from the term $u_0$ in GIRLI, after a certain number of iterations, the images that differ excessively from those being obtained through iterates.}}\\
We use the same numerical example of Test 3. For details on the choice of parameters and some numerical results see Table \ref{tab: table4} and Figure \ref{fig: test4}. \\
In this test, we consider the case where we remove some of the a priori information from the damping term in GIRLI, specifically the data $u^{(i)}$, because they are significantly different from what various iterations provide. This process is applied after a certain number of initial iterations, ensuring that the system has identified at least some elements of the image to be reconstructed. With this approach, we favor using as guesses only images recognized by the method as similar to what we intend to reconstruct. In particular, at each iteration, we calculate the difference between the values of the iteration $u_{k+1}$ and those of the images $u^{(i)}$. For each difference, we compute the Euclidean norm, and if the result exceeds a certain threshold set a priori, the corresponding image is removed from $\sum_{i=1}^{n}(u_k-u^{(i)})$. In Figure \ref{fig: adapt}, we implemented this algorithm: after the tenth iteration, the described comparison was made, with a tolerance set for the norm at $3.2$, eliminating from $\sum_{i=1}^{n}(u_k-u^{(i)})$ the images that do not meet the tolerance criterion, namely for which the norm is greater or equal than $3.2$. The only images that are retained are the four appearing in Figure \ref{fig: adapt}. 
This process is applied at each step, and after some iterations, the method saves only the two elements in the second line of Figure \ref{fig: adapt}.

\begin{figure}[!h]
	\centering
        {\includegraphics[scale=0.35]{images/to_be_reconstructed.png}}\hspace{2cm}
	{\includegraphics[scale=0.35]{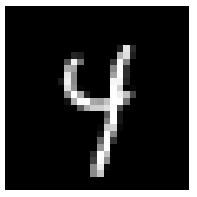}}\hspace{0.5cm}
	{\includegraphics[scale=0.37]{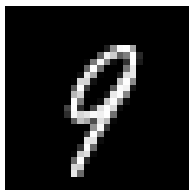}}\hspace{0.5cm}
        {\includegraphics[scale=0.37]{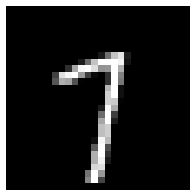}}\hspace{0.5cm}
          {\includegraphics[scale=0.37]{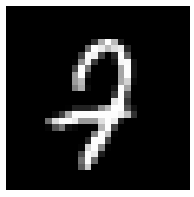}}\vspace{0.5cm}\\
          \includegraphics[scale=0.35]{images/adapt/digit_four.png}\hspace{0.5cm}
	{\includegraphics[scale=0.37]{images/adapt/digit_nine.png}}
	\caption{Test 4. Reconstruction of a digit from the validation set using partial data, removing from the term $u_0$ in GIRLI, after a certain number of iterations, the images that differ excessively from those being obtained through iterates. First line (from the left to the right): image to be reconstructed; images recognized by the system as relevant for reconstruction in the early steps. Second line: the only two images retained by the system until the final iteration.}\label{fig: adapt}
\end{figure}

{	\begin{table}[!h]\caption{Test 4. Left part: Parameters used in the test. Right part: some of the results of the test.}\label{tab: table4}
		\begin{center}
			{\renewcommand{\arraystretch}{1.2}
				\begin{tabular}{|l|c|c|c||c|c|c|}
					\cline{1-7} & & & & & & \vspace{-0.2cm} \\
					Method & $\sigma^2$-noise& $\delta$ & $\tau$ & Iterations & Comp. Time (s) & $\frac{\norm{u_{\textrm{true}} - u_{\textrm{rec}}}_2}{\norm{u_{\textrm{true}}}_2}$\\ \hline  
				{GIRLI-adapt} & 0.03 & 2.78 & 5 & 999 & 13.4658 & 0,4517 \\ \hline 
				\end{tabular}
			}
		\end{center}   
	\end{table} 
}

\begin{figure}[!h]
	\centering
	\includegraphics[scale=0.35]{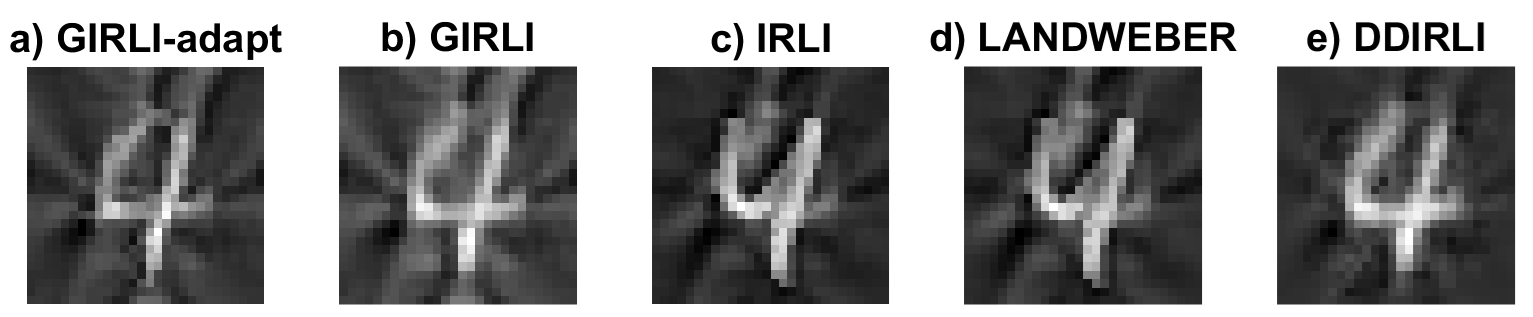}
	\caption{Test 4. Image reconstruction from the validation set. a)-e) Reconstructions using different methods. The adapted version of GIRLI, known as GIRLI-adapt, and DDIRLI produce highly favorable outcomes. However, GIRLI achieves a similar result to DDIRLI but with significantly more iterations and execution time. It is noteworthy that the results of IRLI and Landweber closely resemble the initially chosen guess.}\label{fig: test4}
\end{figure}
 
\subsubsection*{GIRLI vs. GIRLI-GM and an IRLI variant}
We compare the results derived from GIRLI, as per \eqref{eq:GIRLI-Radon}, against those acquired from GIRLI-GM applied to the Radon transform, namely,
\begin{equation}\label{eq:GIRLI-GM-Radon}
\begin{aligned}
&\ukpd := (1-\lambda_k)\ukd - \omega_R R^*\bigl( R\ukd-\yd \bigr) +\lambda^{GM}_k\sqrt[n]{\prod_{i=1}^{n}u^{(i)}},\qquad k\in\mathbb{N}.\\
\end{aligned}
\end{equation}
Following the theoretical framework, we opt for $u_0=\sqrt[n]{\prod_{i=1}^{n}u^{(i)}}$.\\
Next, we extend our analysis to include a comparison with a specific variant of IRLI. In this version, instead of a fixed datum $u^{(i)}$, we incorporate a term that dynamically selects only one datum $u^{(i)}$ at a time, for each iteration $k$, following the condition $i\equiv k \mod n$. 
This guarantees that the selected stabilizing term in IRLI remains dynamic, varying in accordance with the modulo relation, and if the method doesn't converge too early, ensures that, in a recursive manner, all the data are sooner or later incorporated in the iterates, 
\begin{equation}\label{eq:IRLI-Kacz}
\begin{aligned}
&\ukpd := \ukd - \omega_R R^*\bigl( R\ukd-\yd \bigr) -\mu_k (\uk - u^{(i)}),\qquad k\in\mathbb{N},\ \ i\equiv k\mod n,\\
\end{aligned}
\end{equation}
with $\mu_k$ a positive parameter.

We consider three different tests: 
In \textbf{Test 5}, we assume prior knowledge of the digit for reconstruction. As an illustration, let's consider reconstructing the digit $3$. We choose several images $u^{(i)}$ from the MNIST training dataset that represent the digit $3$. A total of $14$ images are selected, and we generate the stabilizing terms for both GIRLI and GIRLI-GM, as depicted in Figure \ref{fig: info_test5}. \\
In \textbf{Test 6}, we show the outcomes of GIRLI and GIRLI-GM using as $u_0$ a common initial guess. \\
Finally, in \textbf{Test 7}, we proceed to compare the outcomes of GIRLI with the modified iterations of IRLI given by \eqref{eq:IRLI-Kacz}.
\begin{figure}[!h]
	\centering
	\includegraphics[scale=0.35]{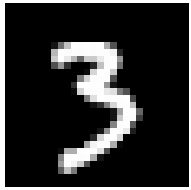}\hspace{3cm}
    \includegraphics[scale=0.35]{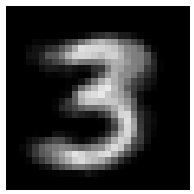}\hspace{1cm}
    \includegraphics[scale=0.35]{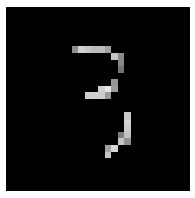}
    \caption{Left image: image to be reconstructed. Moving from the center to the right, we depict the stabilizing terms in GIRLI and GIRLI-GM. These terms are created using 14 images $u^{(i)}$ sourced from the MNIST dataset, specifically by gathering all images representing the number $3$. In the center, there is the stabilizing term $1/14\sum_{i=1}^{14}u^{(i)}$. On the right, we provide the stabilizing term $\sqrt[14]{\prod_{i=1}^{14}u^{(i)}}$. This outcome is not surprising, considering that in the MNIST dataset a substantial number of pixels in the images are zero.}\label{fig: info_test5}
\end{figure}\ \\
\textbf{Test 5} (Figure \ref{fig: test5a} and Figure \ref{fig: test5b}). {\textit{Target: reconstruct a digit from the validation set, comparing GIRLI and GIRLI-GM.}
We implement the iterates \eqref{eq:GIRLI-GM-Radon} and compare the results with GIRLI. Two distinct sets of results are presented. In the first scenario, we assume $\lambda_k=\lambda^{GM}_k=0.01$ for all $k\geq0$, see Figure \ref{fig: test5a} and Table \ref{tab: table5a}. In contrast, the second scenario assumes $\lambda^{GM}_k=0.05$ and $\lambda_k=0.01$ for all $k\geq 0$, see Figure \ref{fig: test5b} and Table \ref{tab: table5b}.
}
\begin{figure}[!h]
	\centering
	\includegraphics[scale=0.3]{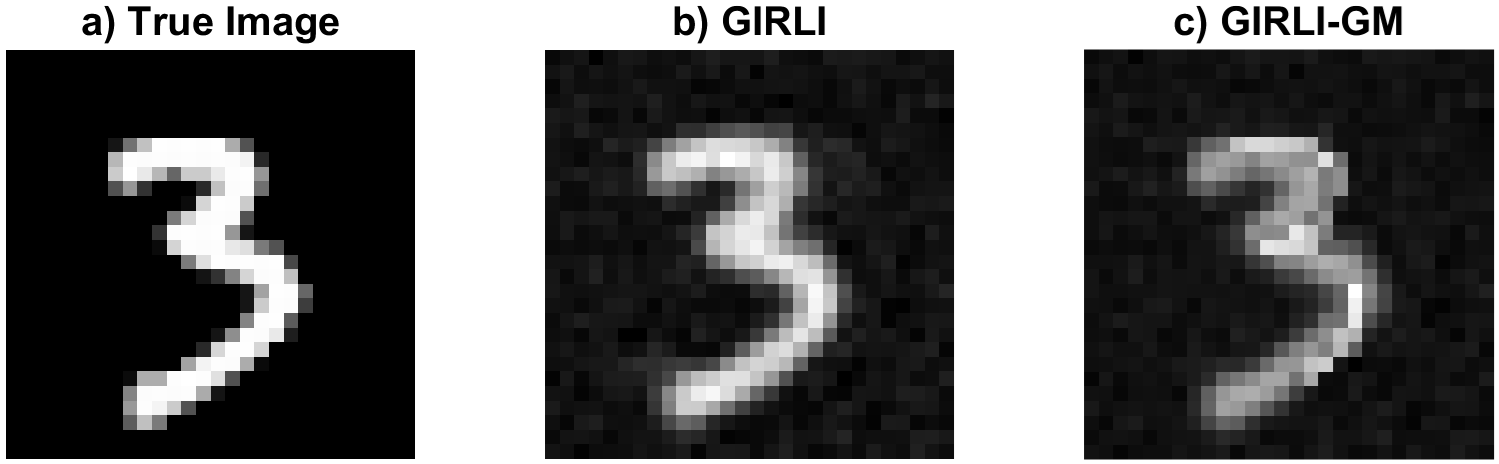}
	\caption{Test 5. Image reconstruction from the validation set. a) Original image to be reconstructed. b)-c) Reconstructions using GIRLI and GIRLI-GM, with $\lambda_k=\lambda^{GM}_k=0.01$. In GIRLI-GM, it is possible to distinguish clearly the contribution provided by the stabilization term.}\label{fig: test5a}
\end{figure}

\begin{figure}[!h]
	\centering
	\includegraphics[scale=0.3]{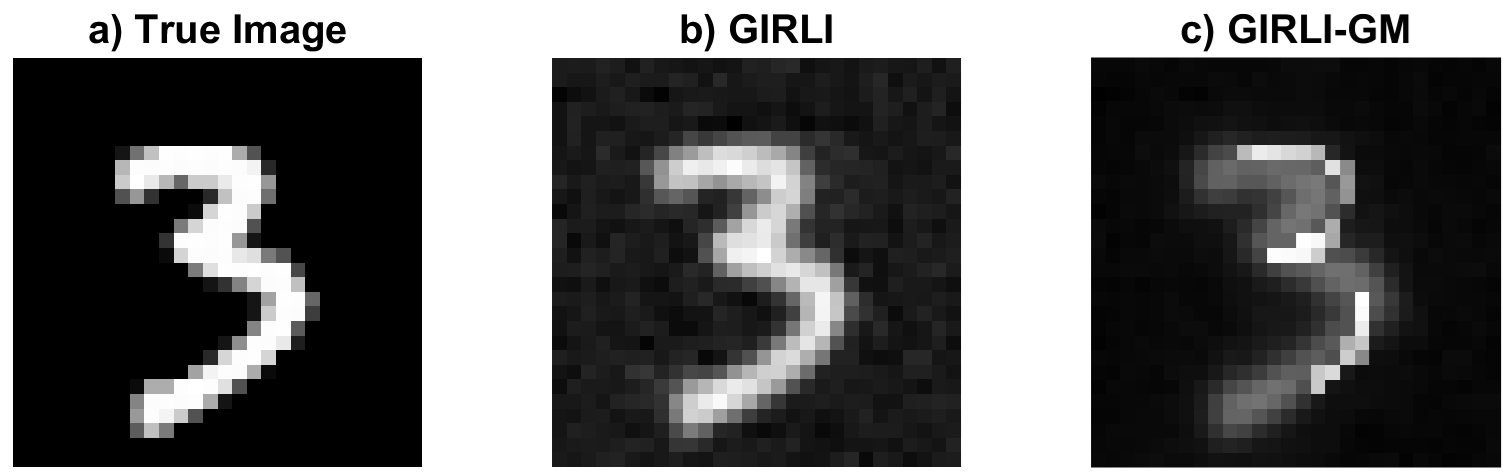}
	\caption{Test 5. Image reconstruction from the validation set. a) Original image to be reconstructed. b)-c) Reconstructions using GIRLI and GIRLI-GM, with $\lambda^{GM}_k=0.05$ and $\lambda_k=0.01$. In GIRLI-GM, the contribution provided by the stabilization method is even more evident compared to the case in Figure \ref{fig: test5a}.}\label{fig: test5b}
\end{figure}
{	\begin{table}[!h]\caption{Test 5. Left part: Parameters used in the test. Right part: some of the results of the test related to Figure \ref{fig: test5a}.}\label{tab: table5a}
		\begin{center}
			{\renewcommand{\arraystretch}{1.2}
				\begin{tabular}{|l|c|c|c||c|c|c|}
					\cline{1-7} & & & & & & \vspace{-0.2cm} \\
					Method & $\sigma^2$-noise& $\delta$ & $\tau$ & Iterations & Comp. Time (s) & $\frac{\norm{u_{\textrm{true}} - u_{\textrm{rec}}}_2}{\norm{u_{\textrm{true}}}_2}$\\ \hline
                    {GIRLI} &  0.5 & 13.3682 & 1.1  & 999 & 9.2136 & 0.1395 \\ \cline{1-1}\cline{5-7}
				{GIRLI-GM} &  &  &  & 999 & 9.0850 & 0.2061 \\ \hline   
				\end{tabular}
			}
		\end{center}   
	\end{table} 
}

{	\begin{table}[!h]\caption{Test 5. Left part: Parameters used in the test. Right part: some of the results of the test related to Figure \ref{fig: test5b}.}\label{tab: table5b}
		\begin{center}
			{\renewcommand{\arraystretch}{1.2}
				\begin{tabular}{|l|c|c|c||c|c|c|}
					\cline{1-7} & & & & & & \vspace{-0.2cm} \\
					Method & $\sigma^2$-noise& $\delta$ & $\tau$ & Iterations & Comp. Time (s) & $\frac{\norm{u_{\textrm{true}} - u_{\textrm{rec}}}_2}{\norm{u_{\textrm{true}}}_2}$\\ \hline
                    {GIRLI} &  0.5 & 13.3682 & 1.1  & 999 & 9.2136 & 0.1395 \\ \cline{1-1}\cline{5-7}
				{GIRLI-GM} &  &  &  & 999 & 8.4266 & 0.3698 \\ \hline   
				\end{tabular}
			}
		\end{center}   
	\end{table} 
}
\textbf{Test 6} (Figure \ref{fig: test6}). {\textit{Target: reconstruct a digit from the validation set, comparing GIRLI and GIRLI-GM, with a common initial guess.}
We implement the iterates \eqref{eq:GIRLI-GM-Radon} and compare the results with GIRLI, employing a shared initial guess for both methods and $\lambda_k=\lambda^{GM}_k=0.01$.
\begin{figure}[!h]
	\centering
	\includegraphics[scale=0.3]{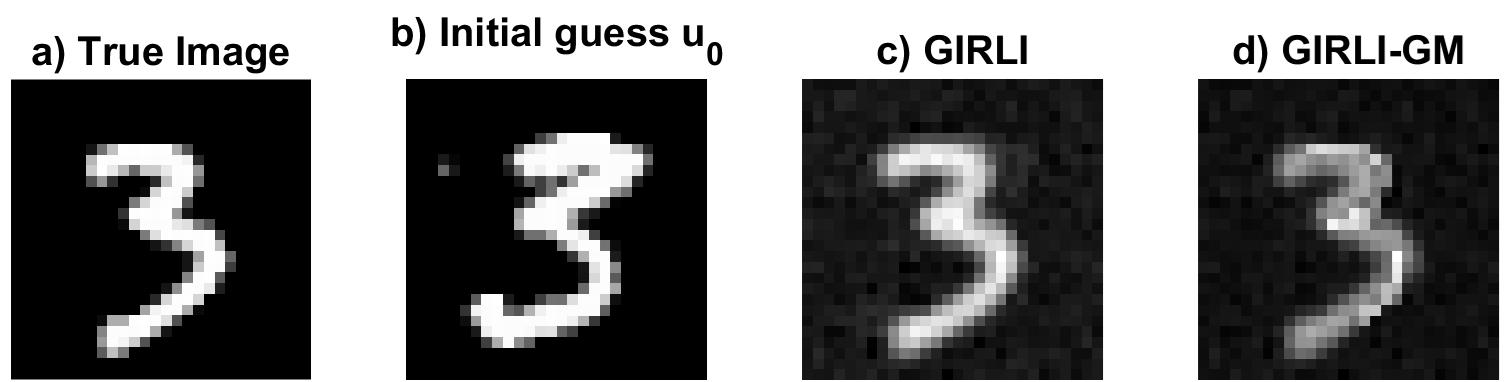}
	\caption{Test 6. Image reconstruction from the validation set. a) Original image to be reconstructed. b) Common initial guess for GIRLI and GIRLI-GM. c)-d) Reconstructions using GIRLI and GIRLI-GM, with $\lambda_k=\lambda^{GM}_k=0.01$.}\label{fig: test6}
\end{figure}

{	\begin{table}[!h]\caption{Test 6. Left part: Parameters used in the test. Right part: some of the results of the test.}\label{tab: table6}
		\begin{center}
			{\renewcommand{\arraystretch}{1.2}
				\begin{tabular}{|l|c|c|c||c|c|c|}
					\cline{1-7} & & & & & & \vspace{-0.2cm} \\
					Method & $\sigma^2$-noise& $\delta$ & $\tau$ & Iterations & Comp. Time (s) & $\frac{\norm{u_{\textrm{true}} - u_{\textrm{rec}}}_2}{\norm{u_{\textrm{true}}}_2}$\\ \hline
                    {GIRLI} &  0.5 & 14.1616 & 1.1  & 999 & 4.7784 & 0.1434 \\ \cline{1-1}\cline{5-7}
				{GIRLI-GM} &  &  &  & 999 & 4.6525 & 0.1960 \\ \hline   
				\end{tabular}
			}
		\end{center}   
	\end{table} 
}

\textbf{Test 7} (Figure \ref{fig: test7}). {\textit{Target: reconstruct a digit from the validation set, comparing GIRLI and the modified version of IRLI given by \eqref{eq:IRLI-Kacz}.}
We implement the iterates \eqref{eq:IRLI-Kacz} with $\mu_k=0.001$ and compare the results with GIRLI, where $\lambda_k=0.01$. In Figure \ref{fig: test7} and Table \ref{tab: table7}, we show the different results.
}

\begin{figure}[!h]
	\centering
	\includegraphics[scale=0.3]{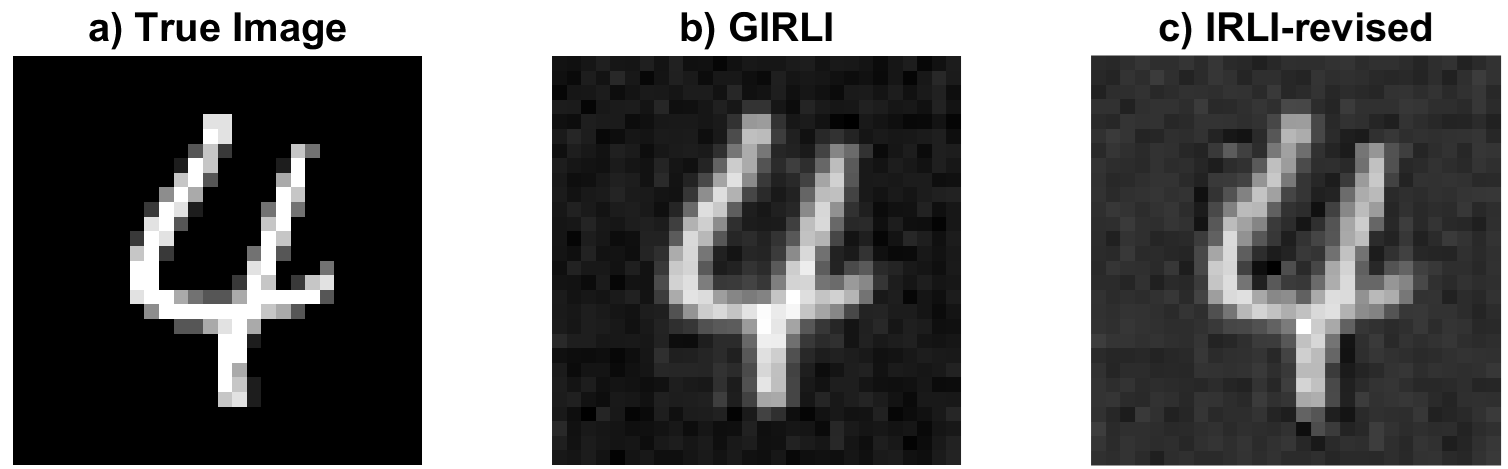}
	\caption{Test 7. Image reconstruction from the validation set. a) Original image to be reconstructed. b)-c) Reconstructions using GIRLI and the modified version of IRLI, with $\lambda_k=0.01$ and $\mu=0.001$.}\label{fig: test7}
\end{figure}

{	\begin{table}[!h]\caption{Test 7. Left part: Parameters used in the test. Right part: some of the results of the test.}\label{tab: table7}
		\begin{center}
			{\renewcommand{\arraystretch}{1.2}
				\begin{tabular}{|l|c|c|c||c|c|c|}
					\cline{1-7} & & & & & & \vspace{-0.2cm} \\
					Method & $\sigma^2$-noise& $\delta$ & $\tau$ & Iterations & Comp. Time (s) & $\frac{\norm{u_{\textrm{true}} - u_{\textrm{rec}}}_2}{\norm{u_{\textrm{true}}}_2}$\\ \hline
                    {GIRLI} &  0.5 & 13.7492 & 1.1  & 999 & 8.7120 & 0.2459 \\ \cline{1-1}\cline{5-7}
				{IRLI-revised} &  &  &  & 101 & 0.8857 & 0.2346 \\ \hline   
				\end{tabular}
			}
		\end{center}   
	\end{table} 
}

\section{Conclusion} \label{sec:conclusion}
In this work, we have analyzed generalized versions of the iteratively regularized Landweber method, both analytically and numerically, demonstrating comparisons with some classical and data-driven algorithms (such as the DDIRLI method presented in \cite{AspBanOktSch20}). Overall, from a numerical perspective, we observed that the DDIRLI method provides very good results with few iterations, while the generalized versions of IRLI require many iterations to achieve comparable results. However, with appropriate precautions, it is possible to expedite these methods, yielding visually impressive outcomes, see for example results of Test 2 and Test 4.
It is important to note that the DDIRLI method is not resource-efficient in terms of memory usage and the amount of data. Therefore, the methods presented in this paper can be regarded as valid alternatives in case issues arise concerning memory constraints and the volume of data in DDIRLI.

\section*{Acknowledgments}
Andrea Aspri is a member of the group GNAMPA (Gruppo Nazionale per l’Analisi Matematica, la Probabilità e le loro Applicazioni) of INdAM (Istituto Nazionale di Alta Matematica). This research has been supported by the project FSE - REACT EU “RICERCA E INNOVAZIONE
2014-2020” (Research and Innovation 2014-2020) and partially performed in the framework of: MIUR-PRIN Grant 2020F3NCPX {\it Mathematics for industry 4.0 (Math4I4)}. \\
Otmar Scherzer's research was funded in whole, or in part, by the Austrian Science Fund (FWF) 10.55776/P34981 -- ``New Inverse Problems of Super-Resolved
Microscopy'' (NIPSUM) and
SFB 10.55776/F68 ``Tomography Across the Scales'', project F6807-N36
(Tomography with Uncertainties). For open access purposes, the author has
applied a CC BY public copyright license to any author-accepted manuscript
version arising from this submission.

The financial support by the Austrian Federal Ministry for Digital and
Economic Affairs, the National Foundation for Research, Technology and
Development and the Christian Doppler Research Association is gratefully acknowledged.

\appendix

%%%%%%%%%%%%%%%%%%%%%%%%%%%%%%
%% References
%%%%%%%%%%%%%%%%%%%%%%%%%%%%%%
\bibliographystyle{plain}
\bibliography{references}

\end{document}